\documentclass[12pt]{amsart}
\usepackage[text={425pt,650pt},centering]{geometry}
\usepackage{graphicx}
\usepackage{epsfig}
\usepackage{amsmath}
\usepackage{amsthm}
\usepackage{amssymb}
\usepackage{color}
\usepackage{caption}
\usepackage{tikz}

\newcommand{\commentout}[1]{}

\newcommand{\Rm}{\mathbb{R}}
\def \Rset {{\mathbb R}}

\def \Zset {{\mathbb Z}}

\def \Tset {{\mathbb T}}

\newcommand{\ra}{\rightarrow}

\newcommand{\nit}{\noindent}

\newcommand{\be}{\begin{equation}}
\newcommand{\ee}{\end{equation}}
\newcommand{\ba}{\begin{eqnarray}}
\newcommand{\ea}{\end{eqnarray}}
\newcommand{\bi}{\begin{itemize}}
\newcommand{\ei}{\end{itemize}}
\newcommand{\br}{\begin{eqnarray}}
\newcommand{\er}{\end{eqnarray}}

\newcommand{\eps}{\varepsilon}

\newcommand{\real}{\mathbb{R}}

\newcommand{\ol}{\overline}
\newcommand{\ep}{\varepsilon}

\newtheorem{theo}{Theorem}[section]

\newtheorem{prop}{Proposition}[section]
\newtheorem{lem}{Lemma}[section]
\newtheorem{cor}{Corollary}[section]
\newtheorem{rmk}{Remark}[section]
\theoremstyle{definition}
\newtheorem{quest}{Question}

\begin{document}
\title[Effective burning velocity]{Inverse problems, non-roundness and flat pieces of the effective burning velocity from an inviscid quadratic Hamilton-Jacobi model}
\author[W. Jing]{Wenjia Jing}
\address{Department of Mathematics\\
The University of Chicago\\ 5734 S. University Avenue Chicago, Illinois 60637, USA}
\email{wjing@math.uchicago.edu}

\author[H. V. Tran]{Hung V. Tran}
\address{Department of Mathematics\\
University of Wisconsin Madison\\ 480 Lincoln drive, Madison, WI 53706, USA}
\email{hung@math.wisc.edu}

\author[Y. Yu]{Yifeng Yu}
\address{Department of Mathematics\\
University of California at Irvine, California 92697, USA}
\email{yyu1@math.uci.edu}

\thanks{
The  work of WJ is partially supported by NSF grant DMS-1515150,
the work of HT is partially supported by NSF grant DMS-1615944,
the work of YY is partially supported by NSF CAREER award \#1151919.
}

\date{}

\keywords{Aubry-Mather theory, effective Hamiltonian, effective burning velocity, flame propagation, homogenization, inverse problems, magnetic Hamiltonian,  reaction-diffusion-convection equation, viscosity solutions}
\subjclass[2010]{
35B27, 
35B40, 
35D40, 
35R30, 
37J50, 
49L25 
}

\maketitle

\begin{abstract}
The main goal of this paper is to understand finer properties of the effective burning velocity from  a combustion model introduced by  Majda and Souganidis \cite{MS_94}.  Motivated by results in \cite{B2} and applications in turbulent combustion,  we show that when the dimension is two and the flow of the ambient fluid is either weak or very strong, the level set of the effective burning velocity has flat pieces. Due to the lack of an applicable Hopf-type rigidity result,  we need to identify the exact location of at least one flat piece.   Implications on the effective flame front and other related inverse type problems are also discussed. 
\end{abstract}

\section{Introduction} 

We consider a flame propagation model proposed by Majda, Souganidis \cite{MS_94} described as follows. 
Suppose that $V:\Rset^n\to  \Rset^n$ is a given smooth, mean zero, $\Zset^n$-periodic and incompressible vector field. 
Let $T=T(x,t):\Rset^n \times [0,\infty) \to \Rset$ be the solution of the reaction-diffusion-convection equation
$$
T_t +  V \cdot D T= \kappa \, \Delta T+\, {1\over \tau_r}\, f(T) \quad \text{in} \ \Rset^n \times (0,\infty)
$$
with given compactly supported initial data $T(x,0)$. 
Here $\kappa$ and $\tau_r$ are positive constants proportional to the flame thickness, which has a small length scale denoted by $\ep>0$. 
The nonlinear function $f(T)$ is of KPP type, i.e.,
\begin{align*}
&f>0 \ \text{in} \ (0,1), \quad f<0 \ \text{in} \ (-\infty,0) \cup (1,\infty),\\
&f'(0)=\inf_{T>0} \frac{f(T)}{T}>0.
\end{align*}
In turbulent combustions, the velocity field usually varies on small scales as well. 
We write $V=V\left({x\over \ep ^{\gamma}}\right)$ and, since the flame thickness is in general much smaller than the turbulence scale, as in \cite{MS_94}, we set $\gamma\in (0,1)$ and write $\kappa=d\ep$ and $\tau_r = \eps$ for some given $d>0$. 
To simplify notations, throughout this paper,  we set $f'(0)=d=1$.
The corresponding equation becomes
$$
T^\ep_t +  V\left({x\over \ep ^{\gamma}}\right) \cdot D T^\ep= \ep \Delta T^{\ep}+\, {1\over \ep}\, f(T^\ep) \quad \text{in} \ \Rset^n \times (0,\infty),
$$
which has a unique solution $T^{\ep}$.   It was proven in \cite{MS_94} that
$T^{\eps} \ra 0$ locally uniformly in $\{(x,t)\,:\,Z < 0\}$, as $\eps \ra 0$, 
and $T^{\eps} \ra 1$ locally uniformly in the interior of
$\{(x,t)\,:\, Z=0\}$. Here, $Z \in C(\real^n\times [0,+\infty))$ is the unique
viscosity solution of a  variational inequality.  Moreover, the set $\Gamma_{t} =\partial \{ x \in \real^n\,:\, Z(x,t) < 0\}$ can be viewed as
a front moving with the normal velocity:
$$
v_{\vec{n}}=\alpha(\vec{n}).
$$
Here, $\alpha$ is the effective burning velocity defined as: For $p \in \Rset^n$,
\be\label{gamma}
\alpha(p)=\inf_{\lambda>0}{1+{\overline H} (\lambda p)\over \lambda}.
\ee
The convex function $\overline {H}:\Rset^n \to \Rset$ is called the effective Hamiltonian.
For each $p\in \Rset^n$, $\overline{H}(p)$ is defined to be the unique constant (ergodic constant) such that the following cell problem
\be\label{Fcell}
H(p+Du,x) = |p + Du|^{2} +V(x)\cdot (p + Du) = {\overline H}(p) \quad \text{in} \ \Tset^n=\Rset^n/\Zset^n
\ee
admits a periodic viscosity solution  $u \in C^{0,1}(\Bbb T^n)$. See \cite{LPV} for the general statement.
As $\gamma<1$, there is no viscous term in \eqref{Fcell} (see \cite[Proposition 1.1]{MS_94}).  Under the level-set approach,   the effective flame front $\Gamma_t$ can be described as the zero level set of $F=F(x,t)$, which satisfies
$$
F_t+\alpha(DF)=0
$$
with $\Gamma_0=\{F(x,0)=0\}$.   Thus, $\alpha(p)$ can be viewed as one way to model turbulent flame speed, a significant quantity in turbulent combustion.  See \cite{EMS, XY} for comparisons between $\alpha(p)$ and the turbulent flame speed modeled by  the G-equation (a popular level-set approach model in combustion community).  
 
The original Hamiltonian $H(p,x) = |p|^2 + V(x) \cdot p$ is similar to the so called Ma\~n\'e Hamiltonian (or  magnetic Lagrangian)  in the dynamical system community. Throughout this paper,  we assume that  $V$ is smooth,  $\Zset^n$-periodic,  incompressible and has mean zero, i.e.,
 \[
 \text{div}(V)=0 \quad \text{and}  \quad \int_{\Bbb T^n}V\,dx=0.
 \]
   It is easy to see that 
$$
\overline H(0)=0,  \quad \overline H(p)\geq |p|^2  \quad \mathrm{and}  \quad  \alpha(p)\geq 2|p|.
$$
Moreover,  $\alpha(p)$ is convex and positive  homogeneous of degree $1$. See Lemma \ref{basic}. 

Practically speaking,  it is always desirable to get more information of the turbulent flame speed (effective burning velocity). In combustion literature, the turbulent flame speed is often considered to be isotropic and various explicit formulas have been introduced to quantify it.  See \cite{Alan, KTW} and the references therein.  From the mathematical perspective,  it is a very interesting and challenging problem to rigorously identify the shape of the effective Hamiltonian or other effective quantities. In this paper,  we are interested in understanding some refined properties of the effective burning velocity $\alpha(p)$. In particular,

\begin{quest} \label{quest-1}
If the flow is not at rest, that is, the velocity field $V$ is not constantly zero, can the convex level set $\{p \in \Rm^n\,:\,\alpha(p)=1\}$ be strictly convex?  A simpler inverse type question is whether  $\alpha(p)=c|p|$  for some $c>0$ (i.e., isotropic)  implies that $V\equiv 0$.   
\end{quest}

\begin{rmk} \label{rem:FlatPiece} When $n=2$,  $\alpha(p)$ is actually $C^1$ away from the origin (Lemma \ref{basic}).  
If the initial flame front is the circle $S^1=\{x \in \Rm^2 \,:\, |x|=1\}$, then, owing to Theorem \ref{frontmove},  the effective front at $t>0$ is $\Gamma_t=\{x+tD\alpha (x) \,:\, x\in S^{1}\}$ which is a $C^1$ and strictly convex curve.   Obviously, if $\alpha$ is Euclidean, that is $\alpha(p) = |p|$, then $\Gamma_t$ is round for all $t > 0$.   If the level curve $\{\alpha(p) = 1\}$ contains a flat piece, then there exist $x_0$, $x_1\in  S^1$ such that
$$
D\alpha(x_s)\equiv  D\alpha(x_0)    \quad \text{for $x_s=(1-s)x_0+sx_1$ and $s\in  [0,1]$ }.
$$
In view of the positive homogeneity of $\alpha$, $D\alpha (p) = D\alpha(x_0) = D\alpha(x_1)$ for all $p \in S^1$ between $x_0$ and $x_1$, i.e. $(p-x_0) \cdot (p-x_1) < 0$. Owing to the representation of $\Gamma_t$, the arc between $x_0$ and $x_1$ of $S^1$ is translated in time and is contained in the front $\Gamma_t$.  This sort of  implies that along the direction   $D\alpha(x_0)$,  the linear transport overwhelms the nonlinear reaction term and dominates the spread of flame, i.e., the propagation behaves like $F_t+D\alpha(x_0)\cdot DF=0$. 

\end{rmk}

Before stating the main results, we review some related works that partly motivate the study of the above questions from the mathematical perspective.  Consider the metric Hamiltonian  $H(p,x)=\sum_{1\leq i,j\leq n}a_{ij}(x)p_ip_j$  with smooth,  periodic and positive definite coefficient $(a_{ij})$.   It was proven in a very interesting paper of Bangert \cite{B2} that, for $n=2$, if the convex level curve $\{p\in \Rset^2\,:\,\overline H(p) \leq 1\}$ is strictly convex, then $(a_{ij})$ must be a constant matrix.  The argument consists of two main ingredients. First,  through a delicate analysis using two dimensional topology, Bangert showed that  if the level set is strictly convex at a point, then the corresponding Mather set of that point is the whole torus $\Tset^2$ and it is foliated by minimizing geodesics pointing to a specific direction. Secondly, a  well-known theorem of Hopf \cite{H} was then applied which says that a periodic Riemannian metric on $\Rset^2$ without conjugate points must be flat.    Part of Bangert's results  (e.g. foliation of the 2d torus by minimizing orbits) was extended to Tonelli  Hamiltonians in \cite{MS} for more general surfaces.  Combining with the Hopf-type rigidity result in \cite{Bialy} for magnetic Hamiltonian, when $n=2$,  it is easy to derive that  the level set of the $\overline H$ associated with the Ma\~ n\'e-type Hamiltonian (\ref{Fcell}) must contain flat pieces unless $V\equiv 0$.  We would like to point out that the non-strict convexity has not been established for general Tonelli Hamiltonian due to the lack of Hopf's rigidity result for Finsler metrics.  See \cite {Z, RGC} for instance.  

The main difficulty in our situation is that the effective burning velocity $\alpha$ is related to the effective Hamiltonian $\overline{H}$ through a variational formulation; see \eqref{gamma}.  In particular, the level set of $\alpha$ is not the same as that of $\overline H$ and Hopf-type rigidity results are not applicable.  In contrast to the proof  in \cite{B2},  we need to figure out the exact location of at least one flat piece in our proofs, which is of independent interest.

In this paper,  we establish some results concerning Question \ref{quest-1} when the flow is either very weak or very strong.  
The first theorem is for any dimension.  

\begin{theo}\label{main1}  Assume that $V$ is not constantly zero.  Then there exists $\ep_0>0$ such that  when $\ep\in  (0,  \ep_0)$, the level curve $S_{\ep}=\{p\in \Rset^n\,:\,\ \alpha_{\ep}(p)=1\}$ is not round (or equivalently, the function $\alpha_\ep$ is not Euclidean).   Here, $\alpha_\ep$ is the effective burning velocity associated with the flow velocity $\ep V$.  
\end{theo}

In two dimensional space, thanks to Lemma \ref{basic}, $\alpha(p)\in  C^1(\Rset^2\backslash\{0\})$.  We  prove further that the level curve of $\alpha$ is not strictly convex under small or strong advections. To state the result precisely, we recall that, for a set $S\subset  \Rset^n$, a point $p$ is said to be {\it a linear  point of $S$}  if there exists a unit vector $q$ and a positive number $\mu_0>0$ such that  the line segment $\{p+tq \,:\,  t\in  [0,  \mu_0]\}\subseteq S$.

\begin{theo}\label{main2}  Assume that $n=2$ and  $V$ is not constantly zero.  Then
\begin{itemize}
\item[(1)] {\upshape(weak flow)} there exists $\ep_0>0$ such that when $\ep\in  (0,  \ep_0)$,  the level curve $S_{\ep}=\{p\in  \Rset^2\,:\,  \alpha_{\ep}(p)=1\}$  contains flat pieces. Here, $\alpha_\ep$ is the effective burning velocity associated with the flow velocity $\ep V$.

\item[(2)] {\upshape(strong flow)} there exists $A_0>0$ such that when $A\geq A_0$,  the level curve $S_{A}=\{p\in  \Rset^2  \,:\, \alpha_{A}(p)=1\}$  contains flat pieces. 
Here, $\alpha_A$ is the effective burning velocity associated with the flow velocity $A V$.
 In particular, if  the flow $\dot \xi=V(\xi)$ has a swirl (i.e.,  a closed orbit that is not a single point),  any $p\in S_A$ which has a rational outward normal vector is a linear point of $S_A$ when $A\geq A_0$. 

\end{itemize}

\end{theo}

We conjecture that the above theorem holds for all amplitude parameters $A \in (0,\infty)$.  So far, we can only show this for some special flows.  Precisely speaking,  

\begin{theo}\label{main3} Assume either
\begin{itemize}
\item[(1)] {\upshape(shear flow)} $V(x)=(v(x_2), 0)$ for $x=(x_1,x_2) \in \Rset^2$, where $v:\Rset \to \Rset$ is a $1$-periodic smooth function with mean zero, or 

\item[(2)] {\upshape(cellular flow)} $V(x)=(-K_{x_2},  K_{x_1})$ with $K(x_1,x_2)=\sin (2\pi x_1) \sin (2\pi x_2)$ for $x=(x_1,x_2) \in \Rset^2$.  

\end{itemize}

\noindent Then for any fixed $A\ne 0$, the level curve $S_{A}=\{p\in  \Rset^2 \,:\,  \alpha_{A}(p)=1\}$  contains flat pieces.  
Here, $\alpha_A$ is the effective burning velocity associated with the flow velocity $A V$.

\end{theo}

We would like to point out that for the cellular flow in part (2) of Theorem \ref{main3},  it was derived by Xin and Yu \cite{XY} that 
$$
\lim_{A\to +\infty}{\log (A)\alpha_A(p)\over A}=C(|p_1|+|p_2|).
$$
for $p=(p_1,p_2)\in  \Rset^2$ and a fixed constant $C$.   See Remark \ref{geofront} for front motion associated with the Hamiltonian $H(p)=|p_1|+|p_2|$.    Moreover,  it remains an interesting question to at least extend the above global result to flows which have both shear and cellular structures, e.g., the cat's eye flow.  A prototypical example is $V(x)=(-K_{x_2},  K_{x_1})$ with $K(x_1,x_2)=\sin (2\pi x_1) \sin (2\pi x_2)+\delta \cos (2\pi x_1) \cos (2\pi x_2)$ for $\delta\in (0,1)$. 

\medskip

{\bf General inverse problems.}  In general, the effective burning velocity cannot determine the structure of the ambient fluid. The reason is that the function  $\alpha(p)$ is homogeneous of degree one and  only depends on  the value of $\overline{H}$ from (\ref{Fcell})  in a bounded domain.  So  $\alpha(p)$ cannot see the velocity field $V$ in places where it rotates very fast.  See Claim 1 in the proof of Theorem \ref{main2}.   See  also (9.5) in \cite{B1} for a related situation.  Nevertheless,  we can address the following  inverse type problem for the effective Hamiltonian $\overline H(p)$.

\begin{quest} \label{quest:main} Assume that $H_i(p,y)=|p|^2 + V_i(y) \cdot p$.  Assume further that $\overline{H}_1=\overline{H}_2$, where $\overline{H}_i$ is the corresponding effective Hamiltonian of $H_i$ for $i=1,2$.
Then what can we conclude about the relations between $V_1$ and $V_2$? Especially, can we identify
some common ``computable" properties shared by $V_1$ and $V_2$?
\end{quest}

This kind of questions was posed and studied first in Luo, Tran and Yu \cite{LTY1} for Hamiltonians of separable forms, i.e., when $H_i(p,y)=H(p)+W_i(y)$ for $i=1,2$.
Here $H(p)$ is the kinetic energy and $W_i$ is the potential energy. As discussed in \cite{LTY1}, a lot of tools from dynamical systems, e.g. KAM theory, Aubry-Mather theory, are involved in the study and the analysis of the problems. Using the approach of ``asymptotic expansion at infinity" introduced in \cite{LTY1},  we can show that if the Fourier coefficients of $V_i$ for $i=1,2$ decay very fast, then
\[
\overline{H}_1=\overline{H}_2 \quad
\Rightarrow \quad \int_{\Tset^n} |V_1|^2\,dy=\int_{\Tset^n}|V_2|^2\,dy.
\]
Since the proof is similar to that of (3) in Theorem 1.2 of \cite{LTY1}, we omit it here.

\subsection*{Outline of the paper}  For readers' convenience,  we give a quick review of  Mather sets and the weak KAM theory in Section 2.  Some basic properties of $\alpha(p)$ (e.g.  the $C^1$ regularity) will be derived as well. 
In Section 3,  we prove Theorems \ref{main1} via perturbation arguments. Theorems \ref{main2} and \ref{main3} will be established in Section 4.  
The use of two dimensional topology is extremely essential here and we do not know yet if the results of Theorems \ref{main2} and \ref{main3} 
can be extended to higher dimensional spaces.

\subsection*{Acknowledgement}
We would like to thank  Sergey Bolotin and Wei Cheng for helpful discussions about Aubry-Mather theory and Hopf's rigidity theorem.    We also want to thank  Rafael Ruggiero for pointing out the  work \cite{Bialy}.  The authors are very grateful to Alan R. Kerstein for helping us understand the notion of turbulent flame speed in combustion literature. 

\section{Preliminaries}   
For the reader's convenience, we briefly review some basic results concerning the Mather sets and the weak KAM theory. See  \cite{W-E, EG, F} for more details.   Let $\Bbb T^n=\Rset^n/\Zset^n$ be the $n$-dimensional   flat torus and $H(p,x)\in C^{\infty}(\Rset^n\times \Rset^n)$ be a Tonelli Hamiltonian, i.e., it satisfies that 
\begin{itemize}
\item[(H1)] (Periodicity) $x\mapsto H(p,x)$ is $\Bbb T^n$-periodic;

\item[(H2)] (Uniform convexity) There exists $c_0>0$ such that for all $\eta=(\eta_1,...,\eta_n)\in \Rset ^n$,
and $(p,x)\in \Rset^n \times \Rset^n$,
$$
\sum_{i,j=1}^{n}\eta_{i}{\partial ^2H\over \partial p_i\partial
p_j}\eta_{j}\geq c_0|\eta|^2.
$$
\end{itemize}
Let $L(q,x)=\sup_{p\in \Rset^n}\{p\cdot q-H(p,x)\}$ be the Lagrangian associated with $H$. Let $\mathcal{W}$ denote the set of all Borel probability
measures on $\Rset^n \times \Bbb T^n$ that are invariant under the corresponding Euler-Lagrangian
flow. 

For each fixed $p\in  \Rset^n$, an element $\mu$ in $\mathcal{W}$ is called a Mather measure if  
$$
\int_{\Rset^n\times \Bbb T^n} (L(q,x)-p\cdot q)\,d\mu=\min_{\nu\in  \mathcal{W}}\int_{\Rset^n\times \Bbb T^n}(L(q,x)-p\cdot q)\,d\nu,
$$
that is, if it minimizes the action associated to $L(q,x) - p\cdot q$. Denote by $\mathcal{W}_p$ the set of all such Mather measures. The value of the minimum action turns out to be $-\overline H(p)$, where $\overline{H}(p)$ is the unique real number such that the following Hamilton-Jacobi equation
\begin{equation}
  \label{cellH}
H(p+Du,x)=\overline H(p) \quad \text{in} \ \Tset^n
\end{equation}
has a periodic viscosity solution $u \in C^{0,1}(\Tset^n)$. 
Equation \eqref{cellH} is usually called the cell problem and $\overline H$ is called the effective Hamiltonian.

The Mather set is defined to be the closure of the union of the support of all Mather measures, i.e., 
$$
\widetilde{\mathcal{M}}_p=\overline {\bigcup_{\mu \in \mathcal{W}_p}\mathrm{supp}(\mu)}.
$$
The projected Mather set $\mathcal {M}_{p}$ is the projection of $\widetilde{\mathcal {M}}_p$ to $\Tset^n$. The following basic and important properties of the Mather set are used frequently in this paper. 

(1) For any viscosity solution $u$ of equation \eqref{cellH}, we have that
\be{}\label{graphsupport}
\widetilde{\mathcal{M}}_p \subset \{(q,x)\in \Rset ^n\times \Bbb
T^n \,:\,
 \text{$Du(x)$ exists and $p+Du(x)=D_{q}L(q,x)$}\}. 
\ee
Moreover $u\in C^{1,1}(\mathcal {M}_{p})$. More precisely, there exists a constant $C$ depending only on $H$ and $p$ such that,
for all $y\in \Tset^n$ and $x \in \mathcal{M}_p$,
\begin{align*}
&|u(y)-u(x)-Du(x)\cdot (y-x)| \leq C|y-x|^2,\\
&|Du(y)-Du(x)| \leq C|y-x|.
\end{align*}

(2) For any orbit $\xi:\Rset \to \Tset^n$ such that $(\dot \xi(t),\xi(t)) \in \widetilde{\mathcal{M}}_p$ for all $t \in \Rset$, we lift $\xi$ to $\Rset^n$ and denote the lifted orbit on $\Rset^n$ still by $\xi$. 
Then, $\xi$ is an absolutely minimizing curve with respect to
$L(q,x) - p\cdot q + \overline H(p)$ in $\Rset ^n$, i.e., for any
$-\infty<s_2<s_1<\infty$, $-\infty<t_2<t_1<\infty$ and $\gamma:[s_2,s_1]\to \Rset ^n$ absolutely continuous
satisfying $\gamma (s_2)=\xi (t_2)$ and $\gamma (s_1)=\xi (t_1)$
the following inequality holds, 
\be\label{abs}
 \int_{s_1}^{s_2} \left( L(\dot \gamma(s),\gamma(s)) +\overline H(p)\right)\,ds \geq 
 \int_{t_1}^{t_2} \left( L(\dot\xi(t),\xi (t))+\overline H(p) \right)\,dt. 
\ee
Moreover, if $\xi$ is a periodic orbit, then its rotation vector
\be\label{directionvector}
{\xi(T)-\xi(0)\over T}\in   \partial \overline H(p).
\ee
Here $T$ is the period of $\xi$ and $\partial \overline H(p)$ is the subdifferential of $\overline H$ at $p$, i.e.,  
$q\in  \partial \overline H(p)$ if  $\overline H(p')\geq \overline H(p)+q\cdot (p'-p)$ for all $p'\in \Rset^n$. 

A central problem in weak KAM theory is to understand the relation between  analytic properties of the effective Hamiltonian $\overline H$ and the underlying Hamiltonian system (e.g. structures of Mather sets).  For instance,  Bangert \cite{B1} gave a detailed characterization of  Mather and Aubry sets on the $2$-torus $\Tset^2$ for metric or mechanical Hamiltonians (i.e., $H(p,x)=\sum_{1\leq i,j\leq n}a_{ij}p_ip_j+W(x)$).   

Let us mention some known results in this direction which are more relevant to this paper.  
As an immediate corollary of   \cite[ Proposition 3]{C},  we have the following result concerning the level curves of $\overline H$ in two dimensional space.
\begin{theo}\label{curve}  
Assume that $n=2$. If  $\overline H(p)=c>\min \overline H$, then the set  $\partial \overline H(p)$ is a closed radial interval, i.e., there exist a unit vector $q\in  \Rset^2$ and $0<s_1\leq s_2$ such that  $\partial \overline H(p)=[s_1q, s_2q]$.   In particular, this implies that the  level set $\{p\in  \Rset^2\,:\,  \overline H(p)=c\}$ is a  closed $C^1$  convex curve and $q$ is the unit outward normal vector at $p$. 
\end{theo}

The following theorem was first proven in  \cite[Theorem 8.1]{EG}. It says that the effective Hamiltonian is strictly convex along any direction that is not tangent to the level set. 

\begin{theo}\label{EGone}   Assume that  $p_1$,  $p_2\in  \Rset^n$.  Suppose that $\overline H(p_2)>\min \overline H$ and $\overline H$ is linear along the line segment connecting $p_1$ and $p_2$.  Then 
$$
\overline H(tp_1+(1-t)p_2)\equiv \overline H(p_2)   \quad \text{for all $t\in  [0,1]$}.
$$
\end{theo}  

In dynamical system literature,  the effective Hamiltonian $\overline H$ and its Lagrangian $\overline L$ are often called {\it $\alpha$ and $\beta$ functions} respectively.  Since $Q\in  \partial \overline H(P)$  $\Leftrightarrow$  $P\in  \partial \overline L(Q)$,   that $\overline L$ is not differentiable at $Q$ implies that $\overline H$ is linear along any two vectors in $\partial \overline L(Q)$.  Accordingly, as an immediate outcome of   \cite[Corollary 1]{MS},  we have that

\begin{theo}\label{foliation}  Let $n=2$, $c>\min \overline H$ and $p\in  \Gamma_c=\{\overline H=c\}$.   If the unit normal vector of $\Gamma_c$ at $p$ is a rational vector and $p$ is not a linear point of  $\Gamma_c$, then  $\mathcal {M}_{p}$ consists of periodic orbits which foliate $\Bbb T^2$. 
\end{theo}


For metric or  mechanical Hamiltonians (i.e.,  $H(p,x)=\sum_{1\leq i,j\leq n}a_{ij}p_ip_j+W(x)$), the above result was first established in  \cite{B2}. 

Finally, we prove some simple properties of the effective burning velocity $\alpha(p)$. 

\begin{lem}\label{basic} Fix $p\in  \Rset^n\backslash\{0\}$.   The followings hold
\begin{itemize}
\item[(1)]  $\alpha:\Rset^n\to \Rset$ is convex.  Also, there exists a unique $\lambda_{p}>0$ such that
$$
\alpha(p)={1+\overline H(\lambda_pp)\over \lambda_p}.
$$
Moreover,    there exists $q\in  \partial \overline H(\lambda_{p}p)$ such that
$$
q\cdot \lambda_{p}p=\overline H(\lambda_{p}p)+1.
$$

\item[(2)] Assume that $n=2$. Then $\alpha(p)\in C^1(\Rset^n\backslash \{0\})$. 

\item[(3)] Assume that $n=2$.  Then $p$ is a  linear point of the level curve $\{\alpha=1\}$ if and only if $\lambda_pp$ is a linear point of  the level curve $\{\overline H=\lambda_p-1\}$.

\end{itemize}
\end{lem}

\begin{proof}
 (1) Since $\overline H(p)\geq |p|^2$, the existence of $\lambda_p$ is clear. For the convexity of $\alpha$, fix $p_0$, $p_1\in  \Rset^n\backslash\{0\}$ and choose $\lambda_0$, $\lambda_1>0$ such that
$$
\alpha(p_0)={1+\overline H(\lambda_0p_0)\over \lambda_0}  \quad \mathrm{and} \quad \alpha(p_1)={1+\overline H(\lambda_1p_1)\over \lambda_1}.
$$
For $\theta\in  [0,1]$,   write $p_{\theta}=\theta p_1+(1-\theta)p_0$.  If $p_{\theta}=0$, the convexity is obvious since $\alpha(0)=0$ and $\alpha(p)\geq 2|p|$.  So we assume $p_\theta\not= 0$.   Choose $\lambda_{\theta}>0$ such that ${1\over \lambda_{\theta}}={\theta\over \lambda_1}+{1-\theta\over \lambda_0}$.  It follows immediately from the definition of $\alpha$ and the convexity of $\overline H$ that 
$$
\alpha(p_\theta)\leq {1+\overline H(\lambda_\theta p_\theta)\over \lambda_\theta}\leq \theta \alpha(p_1)+(1-\theta)\alpha(p_0).
$$
The convexity of $\alpha$ is proved.

Next we prove the uniqueness of $\lambda_p$.   Assume that for $\lambda$, $\bar \lambda>0$, we have that
$$
\alpha(p)={1+\overline H(\lambda p)\over \lambda }={1+\overline H(\bar \lambda p)\over \bar \lambda}.
$$
Then $\partial \alpha(p)\subseteq \partial \overline H(\lambda p)$ and $\partial \alpha(p)\subseteq \partial \overline H(\bar \lambda p)$.  Therefore,   $\partial \overline H(\lambda p)\cap \partial \overline H(\bar \lambda p)\ne\emptyset$.  So $\overline H$ is linear along the line segment connecting $\lambda p$ and $\bar \lambda p$.   Then by Theorem \ref{EGone},  $\overline H(\lambda p)=\overline H(\bar \lambda p)$, which immediately leads to $\lambda=\bar \lambda$.

Next we prove the second equality in Claim (1).   For $\lambda>0$, denote by $w(\lambda)=\overline H(\lambda p)\geq \lambda^2 |p|^2$ and
 \[
 h(\lambda)= {1+w(\lambda)\over \lambda}.
 \]  
 Since $w(\lambda)$ is convex,  there exists a decreasing sequence $\{\lambda_m\}$ such that $\lambda_m \downarrow  \lambda_{p}$ and $w$ is differentiable at $\lambda_m$ and $h'(\lambda_m)\geq 0$.  Clearly,  
\[
w'(\lambda_m)=q_m\cdot  p \quad \text{for any} \ q_m\in  \partial \overline H(\lambda_m p).
\]  
Up to a subsequence, we may assume that $q_m\to q^{+}\in  \partial \overline H(\lambda_pp)$.  
Then  in light of the fact that $h'(\lambda_m)\geq 0$, we deduce
\[
q^{+}\cdot \lambda_{p}p\geq \overline H(\lambda_{p}p)+1.
\]
  Similarly,  by considering an increasing sequence that converges to $\lambda_p$,  
  we can pick $q^{-}\in  \partial \overline H(\lambda_pp)$ such that 
  \[
  q^{-}\cdot \lambda_{p}p\leq \overline H(\lambda_{p}p)+1. 
  \] 
 Since $\partial H(\lambda_pp)$ is a convex set,  we can find $q\in  \partial \overline H(\lambda_{p}p)$ which satisfies 
 \[
 q\cdot \lambda_{p}p=\overline H(\lambda_{p}p)+1.
 \]

(2)   Apparently,  
\begin{equation}
\label{eq:DaDH}
\hat q\in  \partial \alpha(p)  \quad  \Rightarrow  \quad  \hat q\in  \partial \overline H(\lambda_pp).
\end{equation}
Owing to Theorem \ref{curve},  $\partial \alpha(p)$ is also a closed radial interval.  Since $\alpha(p)$ is homogeneous of degree $1$,  any $q\in \partial \alpha(p)$ satisfies $p\cdot q=\alpha(p)$. Since $p \ne 0$ and $\alpha(p) > 0$, this interval  can only contain a single point; it follows that $\alpha$ is differentiable at $p$. 

(3)  ``$\Rightarrow$": This part is true in any dimension.   Clearly,  that $p$ is a linear point of $S=\{\alpha=1\}$ implies that there exists $p'\in S$ such that $p\not= p'$ and 
$$
\partial \alpha(p)\cap \partial \alpha(p')\not= \emptyset.
$$
By (\ref{eq:DaDH}),  $\partial \alpha(p)\subseteq  \partial \overline H(\lambda_pp)$ and  $\partial \alpha(p')\subseteq  \partial \overline H(\lambda_{p'}p')$.  Hence $\overline H$ is linear along the line segment connecting $\lambda_pp$ and $\lambda_{p'}p'$.   Then Theorem \ref{EGone}  implies that $\overline H(\lambda_p p)=\overline H(\lambda_{p'}p')$ and $\lambda_p=\lambda_{p'}$.  The necessity then follows.  

Now we prove the sufficiency which relies on the $2$-dimensional topology. For $p\in  \Rset^2$,  assume that  $\lambda_pp$ is a linear point of  the level curve $C_p=\{\overline H=\lambda_p-1\}$, i.e.,  there exists a distinct vector $\lambda_pp'\in C_p$ such that the line segment $l_{p,p'} = \{s p + (1-s)p' \,:\, s \in [0,1]\}$, which connects $p$ and $p'$, satisfies $l_{p,p'}\subset \{G=1\}$.  Here  for $q\in  \Rset^2$, 
$$
G(q)={1+\overline H(\lambda_pq)\over \lambda_p}\geq \alpha(q).
$$
By Theorem \ref{curve} and $D\alpha(p)\in \partial G(p)=\partial \overline H(\lambda_pp)$,  we have that 
$$
\partial G(p)=\{sD\alpha(p) \,:\,   s\in  [\theta_1,  \theta_2]\}
$$
for some $0<\theta_1\leq \theta_2$. Therefore $D\alpha(p)\cdot (p'-p)=0$, which implies that 
$$
1=G(q)\geq \alpha(q)\geq \alpha(p)+D\alpha(p)\cdot (q-p)=\alpha(p)=1
$$
for any $q\in  l_{p,p'}$.  Hence $l_{p,p'}\subset \{\alpha=1\}$ and $p$ is a linear point.  

\end{proof}

The following result characterizes the shape of the moving front when the initial front is the unit circle in $\Rset^2$.  

\begin{theo}\label{frontmove} 
Suppose that $n=2$ and $\alpha : \Rset^2 \to  \Rset$ is convex, coercive and positive homogeneous of degree $1$.   
Let $u\in C(\Rset^2\times [0, +\infty))$ be the unique viscosity solution to
$$
\begin{cases}
u_t + \alpha (Du)=0\quad \text{in $\Rset^n\times (0,  +\infty)$}\\
u(x,0)=|x|-1.
\end{cases}
$$
Then $u(x,t)=\max \{-t \alpha (p)+x\cdot p \,:\, |p|\leq 1\}-1$ and its zero level set is
\be\label{levelformu}
\Gamma_t :=\{x \in \Rm^2 \,:\, u(x,t)=0\}=\{p+tq \,:\,  p \in S^1,\  q\in  \partial \alpha (p)\}.
\ee
Also $\Gamma_t$ is $C^1$.  Moreover, 
\be\label{duality}
\text{$\alpha\in C^1(\Rset^2\backslash\{0\})$   $\Longleftrightarrow$   $\Gamma_t$ is strictly convex}. 
\ee
\end{theo}

\begin{proof} We first prove the representation  (\ref{levelformu}).  Due to the $1$-homogeneity of $\alpha(p)$,   $p\cdot q=\alpha(p)$ for any $q\in  \partial \alpha(p)$.  The formula of $u(x,t)$ then follows directly from Theorem 3.1 in \cite{BE}. Clearly,  if $u(x,t)>-1$,  then
$$
u(x,t)=\max \{-t \alpha (p)+x\cdot p \,:\, |p|= 1\}-1.
$$
Now fix $x\in  \Rset^2$ such that $u(x,t)=0$.  Choose $|\bar p|=1$ such that
\begin{equation}\label{s-1}
u(x,t)=\ol{p}\cdot x-t\alpha(\ol{p})-1 = 0.
\end{equation}
By the Lagrange multiplier method, we get $x-tq=s \ol{p}$ for some $q\in  \partial \alpha(\ol{p})$ and some $s\in \Rset$.  
We use \eqref{s-1} to deduce further that $s=1$, and hence $x=\ol{p}+tq$.

Conversely,  if $x=\ol{p}+tq$ for some $\ol{p} \in S^1$ and $q\in \partial \alpha(\ol{p})$,  we want to show that $u(x,t)=0$. 
In fact,  in the representation formula of $u$, choosing $p=\bar p$ immediately leads to $u(x,t)\geq 0$.  
On the other hand,  for any $|p|=1$,  $q\in \partial \alpha(\ol{p})$ implies
$$
\alpha(p)\geq \alpha(\bar p)+q\cdot(p-\bar p).
$$
Therefore  
\[
-t\alpha(p)+x\cdot p\leq -t\alpha(\bar p)-tq\cdot(p-\bar p)+x\cdot p=p\cdot \bar p\leq 1.
\]
 So $u(x,t)\leq 0$. Hence we proved that $u(x,t)=0$.

\medskip

Next we show that $\Gamma_t$ is $C^1$.  Fix $t>0$.   Owing to the above arguments,  given $x\in \Gamma_t$, there  exists a unique unit  vector  $p_x$ such that $x=p_x+q_x$ for some $q_x\in  \partial \alpha(p_x)$ and 
$$
u(x,t)=-t\alpha (p_x)+p_x\cdot x-1.
$$
The uniqueness is due to the convexity of $\alpha$ which implies that  $(p-p')\cdot (q-q')\geq 0$ for $q\in  \partial \alpha(p)$ and $q'\in  \partial \alpha(p')$. Hence $x\to p_x$ is a continuous map from $\Gamma_t$  to the unit circle.  Combining with  $p_x\in  \partial_x u(x,t)$,   $p_x$ is the outward unit normal vector of $\Gamma_t$ at $x$ and $\Gamma_t$ is $C^1$. 

Next we prove the duality (\ref{duality}).  Again fix $t>0$.      This direction ``$\Leftarrow$" follows immediately from the representation formula (\ref{levelformu}).   So let us prove ``$\Rightarrow$".   We argue by contradiction.  Assume that $\alpha$ is $C^1$ away from the origin.   If $\Gamma_t$ is not strictly convex,  then there exist $x$, $y\in \Gamma_t$ such that $x\not= y$ and  $p_x=p_y$.  Hence $q_x\not= q_y$.  However,  $q_x=D\alpha(p_x)=D\alpha(p_y)=q_y$, which is a contradiction.  This proves that \eqref{duality} holds.  
\end{proof}

\begin{rmk}\label{geofront} As mentioned in Remark \ref{rem:FlatPiece}, when $n=2$,  a flat piece on the level set $\{\alpha(p) =1\}$ leads to a translated arc of the unit circle on $\Gamma_t$.  Moreover, singular points of $\alpha$ (i.e.,\/ points where $\partial \alpha(p)$ contains a line segment) generate flat pieces on $\Gamma_t$.   For example,  if $\alpha(p)=|p_1|+|p_2|$ for $p=(p_1,p_2)$, then the front $\Gamma_1$ at $t = 1$ is the closed curve shown in Fig. \ref{fig:flame}:

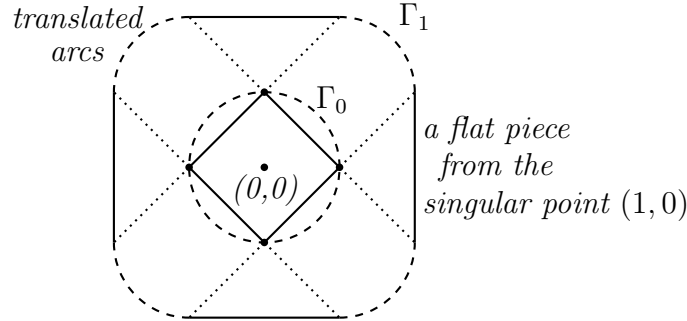
\begin{figure}[h]
\begin{center}
\begin{tikzpicture}
\draw[thick,dashed] (-1,2) arc (90:180:1);
\draw[thick,dashed] (-2,-1) arc (180:270:1);
\draw[thick,dashed] (1,-2) arc (270:360:1);
\draw[thick,dashed] (2,1) arc (0:90:1);
\draw[thick] (1,2) -- (-1,2);
\draw[thick] (1,-2) -- (-1,-2);
\draw[thick] (2,1) -- (2,-1);
\draw[thick] (-2,1) -- (-2,-1);
\draw[thick] (1,0) -- (0,1);
\draw[thick] (0,1) -- (-1,0);
\draw[thick] (-1,0) -- (0,-1);
\draw[thick] (0,-1) -- (1,0);
\fill (1,0) circle (0.05);
\fill (0,1) circle (0.05);
\fill (-1,0) circle (0.05);
\fill (0,-1) circle (0.05);
\fill (0,0) circle (0.05);
\draw[thick,dashed] (0,0) circle (1);
\draw[thick,dotted] (1,0) -- (2,1);
\draw[thick,dotted] (1,0) -- (2,-1);
\draw[thick,dotted] (0,1) -- (1,2);
\draw[thick,dotted] (0,1) -- (-1,2);
\draw[thick,dotted] (-1,0) -- (-2,1);
\draw[thick,dotted] (-1,0) -- (-2,-1);
\draw[thick,dotted] (0,-1) -- (-1,-2);
\draw[thick,dotted] (0,-1) -- (1,-2);
\draw(0,-0.3) node {(0,0)};
\draw (0.9,0.9) node {$\Gamma_0$};
\draw (2,2) node {$\Gamma_1$};
\draw (-2.5,2) node {translated};
\draw (-2.5,1.5) node {arcs};
\draw (3,0.5) node {\ a flat piece};
\draw (3,0) node {\ from the};
\draw (3.8,-0.5) node {\ singular point $(1,0)$};

\end{tikzpicture}
\caption{\label{fig:flame}Front propagation and the shape of  $\Gamma_1$.}
\end{center}
\end{figure}

\end{rmk}

\section{The Proof of Theorem \ref{main1}}

Fix $p\in \Rset^n$ to be an irrational vector satisfying a Diophantine condition, i.e., there exist $c=c(p)>0$ and $\gamma>0$ such that
\[
|p\cdot k| \geq \frac{c}{|k|^\gamma} \quad \text{for all} \ k \in \Zset^n \setminus \{0\}.
\]
For small $\ep$,  let  $\overline{H}_\ep(p)$ be the effective Hamiltonian associated with $|p|^2+\ep V\cdot p$, i.e.,
\begin{equation}\label{aa}
|p+Du^\ep|^2 + \ep V\cdot (p+Du^\ep) = \overline{H}_\ep(p).
\end{equation}
We now perform a formal asymptotic expansion  in term of $\ep$, which will be proved rigorously by using the viscosity solution techniques.
Suppose that
\begin{align*}
u^\ep &= \ep \phi_1 + \ep^2 \phi_2+ \cdots\\
\overline{H}_\ep(p) &= a_0(p) + \ep a_1(p) + \ep^2 a_2(p) + \cdots
\end{align*}
We then get that
\begin{align} \label{recursive}
a_0(p)&=|p|^2\\
a_1(p)&=2p\cdot D\phi_1 + V\cdot p \Rightarrow a_1(p)=0 \notag\\
a_2(p)&=2p\cdot D\phi_2 + |D\phi_1|^2 + V \cdot D\phi_1 \Rightarrow a_2(p)= \int_{\Tset^n} |D\phi_1|^2\,dx \notag\\
&\cdots \notag
\end{align}
Set
\[
V = \sum_{k \neq 0} v_k e^{i 2\pi k\cdot x}.
\]
We need $v_k \cdot k=0$ for all $k$ to have that $\text{div} V=0$. Then we get
\[
D\phi_1 =-\frac{1}{2} \sum_{k \neq 0} \frac{ (p\cdot v_k) e^{i 2\pi k\cdot x} k}{p\cdot k},
\]
and
\[
a_2(p)= \frac{1}{4} \sum_{k \neq 0} \frac{ |p\cdot v_k|^2 |k|^2}{|p\cdot k|^2}.
\]
Thus, formally, we can conclude that
\[
\overline{H}_\ep(p) \approx |p|^2 + \ep^2  \frac{1}{4} \sum_{k \neq 0} \frac{ |p\cdot v_k|^2 |k|^2}{|p\cdot k|^2} + O(\ep^3).
\]
We now prove this expansion formula rigorously.  See related computations in \cite{G2003}. 
\begin{lem}\label{lem:exp-r}
There exists $\tau > 0$, such that for all $|p|\in  \left[\tau,  {1\over \tau}\right]$, we have\begin{equation}\label{exp-r}
\overline{H}_\ep(p) = |p|^2 + \ep^2  \frac{1}{4} \sum_{k \neq 0} \frac{ |p\cdot v_k|^2 |k|^2}{|p\cdot k|^2} + O(\ep^3)
\end{equation}
as $\ep \to 0$. Here, the error term satisfies $|O(\ep^3)|\leq K\ep ^3$ for some $K$ depending only on $\tau$, $V$ and ${p\over |p|}$.  
\end{lem}

\begin{proof}  
As $p$ satisfies a Diophantine condition, we are able to solve the following two equations explicitly in $\Tset^n$ by computing Fourier coefficients
\[
\begin{cases}
p \cdot D\phi_1 = -\frac{1}{2} V\cdot p\\
p\cdot D\phi_2 = \frac{1}{2} \left( a_2(p) - |D\phi_1|^2 - V\cdot D\phi_1 \right).
\end{cases}
\]
Here $\phi_1, \phi_2:\Tset^n \to \Rset$ are unknown functions. 

Set $w^\ep=\ep \phi_1 + \ep^2 \phi_2$. Then, in light of the properties of $\phi_1, \phi_2$, $w^\ep$ satisfies
\[
|p+Dw^\ep|^2 + \ep V\cdot (p+Dw^\ep) = |p|^2 + \ep^2 a_2(p) + O(\ep^3).
\]
By looking at places where $v^\ep-w^\ep$ attains its maximum and minimum and using the definition of viscosity solutions, we derive that
\[
\overline{H}_\ep(p)= |p|^2 + \ep^2  a_2(p) + O(\ep^3).
\]
The error estimate can be read from the proof easily.
\end{proof}

It is obvious that $a_2(p)$ is not a constant function of $p$.   Hence Theorem \ref{main1} follows immediately from the following lemma. 

\begin{lem}\label{lem:MS-pert}
For any vector $p$ satisfying a Diophantine condition,
\begin{equation}\label{alpha-ep-lim}
\lim_{\ep \to 0} \frac{\alpha_\ep(p)-2|p|}{\ep^2 |p|}=a_2(p).
\end{equation}
\end{lem}

\begin{proof}  
Since $\alpha_{\ep}(p)$ is homogeneous of degree 1 and $a_2(p)$ is homogeneous of degree 0 ($a_2(p)=a_2(\lambda p)$ for all $\lambda>0$), we may assume that $|p|=1$. 
Thanks to Lemma \ref{lem:exp-r}, we can write 
\be\label{asy}
\overline{H}_\ep(p)= |p|^2 + \ep^2  \frac{1}{4} \sum_{k \neq 0} \frac{ |p\cdot v_k|^2 |k|^2}{|p\cdot k|^2} + O(\ep^3) = 1+ \ep^2 a_2(p) + O(\ep^3).
\ee
Owing to Lemma \ref{basic},  there exists a unique constant $\lambda_{\ep}=\lambda_{\ep}(p)>0$ such that
$$
\alpha_{\ep}(p)={1+\overline{H}_\ep(\lambda_{\ep}p)\over \lambda_\ep}.
$$
It is easy to see that $\lambda_{\ep}\to  1$ as $\ep\to 0$.   So by  Lemma \ref{lem:exp-r} and $\lambda_{\ep}+{1\over \lambda_{\ep}}\geq 2$,
$$
\alpha_{\ep}(p)\geq 2+\ep ^2a_2(p)+O(\ep^3).
$$
Also,  by the definition,  it is obvious that 
\[
\alpha_{\ep}\leq 1+\overline H_{\ep}(p)=2+\ep^2 a_2(p) + O(\ep^3).
\]  
Hence the conclusion of the lemma holds. 
\end{proof}

\section{Proofs of Theorems \ref{main2} and \ref{main3}}

Before proceeding to the proofs,  we would like to point out something crucial in the arguments.  Clearly,  if the level curve $\{\alpha=1\}$ of $\alpha$ is strictly convex at $p$,   the level curve $\{\overline H=\lambda_{p}-1\}$ of $\ol{H}$ must be strictly convex at $\lambda_{p} p$, where $\lambda_p p$ is determined by Lemma \ref{basic}.  Nevertheless, our results do not follow  from any rigidity result for $\ol{H}$, namely that of \cite{Bialy}. Indeed, different $p$ in $\{\alpha = 1\}$ might correspond to different $\lambda_p$, which corresponds to different energy levels of $\ol{H}$, but the rigidity result from \cite{Bialy} can be applied only on the same energy level.  A key point of our proofs is to identify the exact location of at least one flat piece.

We first prove Claim (1)  of  Theorem \ref{main2}.  

\begin{proof}[{\bf Proof of Theorem \ref{main2} (1)}]
We carry out the proof in a few steps.

\medskip

  {\bf Step 1}:  Due to Claim (2) of Lemma \ref{basic}, the level curve $S_{\ep}$ is $C^1$.
  It is worth keeping in mind that $\alpha_\ep(p)$ is homogeneous of degree $1$.
 For each $p \in S_\ep$, denote $n_p$ the outward unit normal vector at $p$ to $S_\ep$.  

\medskip

{\bf Step 2:}  If $V$ is not  constantly zero,   by Lemma \ref{rigid} below,  there exists $x_0\in  \Rset^2$ and a unit rational vector $q_0$ and $T>0$ such that $Tq_0\in  \Zset^2$ and 
$$
\int_{0}^{T}q_0\cdot DV(x_0+q_0t)\,dt\ne 0.
$$
Here $q\cdot DV=D(q\cdot V)$.

\medskip

{\bf Step 3:}  For each $\ep>0$,  choose $p_{\ep}\in  S_{\ep}$ such that $n_{p_\ep}=q_0$.  To simplify notations,  we write $n_{\ep}=n_{p_\ep}$.  We claim that when $\ep$ is small enough, $p_{\ep}$ is a linear point of the set $\{\alpha_{\ep}=1\}$. Suppose this is false,  then there exists a decreasing sequence $\ep_m \downarrow 0$ and as equence $\{p_{\ep_m}\}$ such that $p_{\ep_m}$ is not a linear point of the set $\{\alpha_{\ep_m}=1\}$.  By (3) of Lemma \ref{basic},  $\tilde p_{\ep_m}=\lambda_{\ep_m}p_{\ep_m}$ is not a linear point of  the level curve $\{\overline H_{\ep_m}=\lambda_{\ep_m}-1\}$ either. Here  $\lambda_{\ep_m}>0$ is from Lemma \ref{basic}. Clearly,  the outward unit normal vector of the level curve $\{\overline H_{\ep_m}=\lambda_{\ep_m}-1\}$ at $\tilde p_{\ep_m}$ is also $q_0$.  Accordingly toTheorem \ref{foliation},  the projected Mather set $\mathcal {M}_{\tilde p_{\ep_m}}$ is the whole torus $\mathbb {T}^2$. Moreover, by (\ref{directionvector}),  there is a periodic minimizing orbit $\xi_{m}:\Rset\to  \Rset^2$  passing through $x_0$ from Step 2 such that $\xi_{m}(0)=x_0$,  $\xi_{m}(t_{m})=x_0+Tq_0$ for some $t_m>0$ and $\xi$ satisfies the Euler-Lagrange equation associated with the Lagrangian $L(q,x)={1\over 4}|q-V|^2$:
$$
{d(\dot \xi_m(t)-\ep_m V(\xi_m(t)))\over dt}=-(\dot \xi_m(t)-\ep_m V(\xi_m(t))\cdot \ep_m DV(\xi_m).
$$
Taking the integration on both sides over $[0,t_m]$,  and by periodicity, we get
$$   
\int_{0}^{t_m} (\dot \xi_m(t)-\ep_m V(\xi_m(t))\cdot DV(\xi_m)\,dt=0.
$$
Sending $m\to  +\infty$, we find
$$
\int_{0}^{T}q_0\cdot DV(x_0+q_0t)\,dt=0.
$$
This contradicts to Step 2. As a result, we identified a flat piece of $S_\ep$.
\end{proof}

\begin{lem}\label{rigid}  If 
$$
\int_{0}^{T}q\cdot DV(x+qt)\,dt=0.
$$
for any $x\in  \Rset^2$ and any rational unit vector $q\in  \Rset^2$ and $Tq\in  \Zset^2$,  then 
$$
V\equiv 0
$$
(Caution:   $q\cdot DV(x+qt)=D(q\cdot V(x+qt))\not=  {dV(x+qt)\over dt}$).  
\end{lem}

\begin{proof}
Assume that $V(y)=\sum_{k\in  \Zset^2} v_k e^{i 2\pi k\cdot y}$, where $\{v_k\} \subset  \Rset^2$ are Fourier coefficients of $V$.   Since $\text{div}(V)=0$ and $\int_{\Bbb T^2}V\,dx=0$,  we have that $v_0=0$ and 
\be\label{van}
k\cdot v_k=0
\ee
 for all $k$.    Then for any $q\in  \Rset^2$,  
 \[
   q\cdot DV(y)=D(q\cdot DV)=2\pi i\sum_{k\in  \Zset^2\setminus\{0\}} (q\cdot v_k) e^{i2\pi k\cdot y} k. 
 \]  
Now fix an integer vector $k=(l_1,l_2) \in \Zset^2 \setminus \{0\}$. Choose $q=(-l_2,l_1)$ and $T=1$.   Then
$$
\int_{0}^{1}q\cdot DV(x+qt)\,dt=0\quad \Rightarrow \quad q\cdot v_{k}  e^{i 2\pi k\cdot x}+ q\cdot v_{-k}  e^{-i 2\pi k\cdot x}=0
$$
for any $x\in  \Rset^2$.   So $q\cdot v_{k}= q\cdot v_{-k} =0$,  which implies that  $v_k=\beta_k k $ for some $\beta_k\in  \Rset$.  Accordingly,  (\ref{van}) leads to $ \beta_k=0$ for all $k\in \Zset^2$.  So $V\equiv 0$. 
\end{proof}

Next we prove Claim (2) of Theorem \ref{main2}. 

\begin{proof}[{\bf Proof of Theorem \ref{main2}(2)}]

Recall that, for $A\geq 0$ and $p\in \Rset^2$,  $\alpha_A(p)$ is defined as
 $$
\alpha_{A}(p)=\inf_{\lambda>0}{\overline H_{A}(\lambda p)+1\over \lambda}.
$$
Here  $\overline H_{A}$ is the effective Hamiltonian associated with  $H_A(p,x)=|p|^2+A V\cdot p$.   The corresponding Lagrangian is 
$$
L_A(q,x)={1\over 4}|q-AV(x)|^2.
$$
For $p\in  \Rset^2\backslash\{0\}$ and $A\geq 0$,  by Lemma \ref{basic},  denote $\lambda_{p,A}>0$ as the unique positive number such that
$$
\alpha_A(p)={\overline H_A(\lambda_{p,A}p)+1\over \lambda_{p,A}}.
$$

 Let $K$ be a stream function such that $V=(-K_{x_2}, K_{x_1})$.  Clearly,  we have that $DK\cdot V\equiv 0$.  We consider the dynamical system $\dot{\xi} = V(\xi)$.  By Poincar\'e recurrence theorem, we have the following  two cases.

\medskip

 {\bf Case 1:} $\dot \xi=V(\xi)$ has a non-critical closed periodic orbit on $\Rset^2$.   
By  the stability in 2d, there exists a strip of closed periodic orbits in its neighborhood.   Without loss of generality,  we may label them as  $\gamma_{s}(t)$ for $s\in  [0, \delta]$ for some $\delta>0$ sufficiently small such that $K(\gamma_s(t))\equiv s$ and $\gamma_s(0)=\gamma_s(T_s)$ for some $T_s>0$ (minimum period).  See the following figure.  Denote  $\Gamma=\cup_{s\in  [0,  \delta]}\gamma_s$ and 
$$
\tau=\max_{x\in  \Gamma}|DK(x)|>0.
$$


\begin{figure}[h]
\begin{center}
\begin{tikzpicture}
\draw  plot [smooth cycle] coordinates {(0,0) (-0.5,1) (-2,2) (0,3) (1.5,2) (1,1)};
\draw (0,2.8) node {$\gamma_0$};
\draw[->] (0,3)--(0.05,3);

\draw  plot [smooth cycle] coordinates {(0,-1) (-0.5,0.2) (-2,1) (-3,2) (-2,4) (0,4.5) (2,3) (3,2) (1.3,0)};
\draw (0,4.3) node {$\gamma_s$};
\draw[->] (-0.05,4.5)--(0,4.5);

\draw  plot [smooth cycle] coordinates {(0,-2) (-0.5,-1) (-2,0) (-3,1) (-4,2) (-3,3.5) (-2,5) (0,5.5) (2,4) (3,3) (4, 2.5) (3,1) (2,0)};
\draw (0,5.3) node {$\gamma_\delta$};
\draw[->] (-0.05,5.5)--(0,5.5);

\draw (2,3) node {$\Gamma$};

\draw plot [smooth] coordinates {(-3,-1) (-2,0) (-0.5,1) (1,1) (2,0) (3,-1)};
\draw (3,-0.8) node {$\xi$};

\fill (-2,0) circle (0.05);
\draw (-2,-0.3) node {$t_1$};

\fill (-0.5,1) circle (0.05);
\draw (-0.5,1.3)node {$t_2$};

\fill (1,1) circle (0.05);
\draw (1,1.3) node {$t_3$};

\fill (2,0) circle (0.05);
\draw (2,-0.3) node {$t_4$};
\end{tikzpicture}
\caption{Closed periodic orbits in $\Gamma$}
\end{center}
\end{figure}
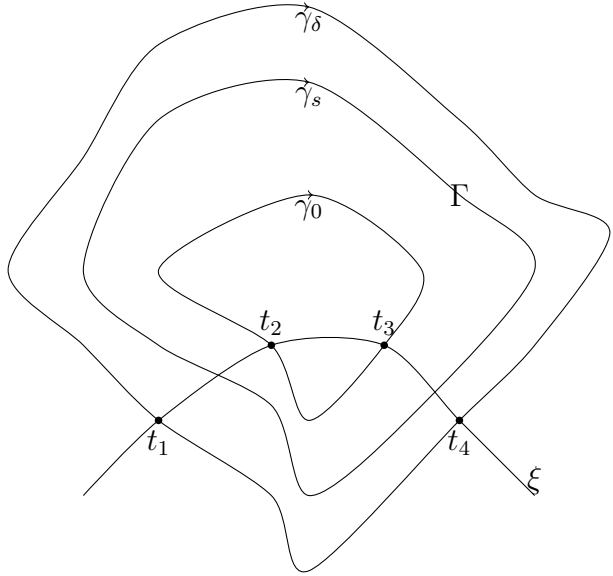

\noindent {\bf Claim 1:} For $p\in  \Rset^2$,    if  $\overline H_A(p)<\bar c A^2$ for $\bar c=  {4\delta^2\over T_{\delta}^{2} \tau^2 }$,  then any unbounded absolutely minimizing trajectory associated with $L_A+\overline H_A(p)$ cannot  intersect   $\gamma_0$.   

We argue by contradiction.  If not,   let $\xi:\Rset\to  \Rset^2$ be an unbounded absolutely minimizing trajectory with $\xi\cap \gamma_0\not=\emptyset$.  Then there must exist $t_1< t_2\leq t_3<t_4$ such that 
\[
\xi(t_1), \xi(t_4)\in   \gamma_{\delta},  \quad \xi(t_2), \xi(t_3)\in  \gamma_{0} \quad \text{and} \quad \xi([t_1,t_2])\cup  \xi([t_3,t_4])\subset \Gamma.
\]
See Figure 2 for demonstration.  
Set
\[
E_1=\int_{t_1}^{t_2}{1\over 4} |\dot \xi-AV(\xi)|^2+\overline H_A(p)\,ds \quad \text{and} 
\quad E_2=\int_{t_3}^{t_4}{1\over 4} |\dot \xi-AV(\xi)|^2+\overline H_A(p)\,ds.
\]
Since 
$$
|\dot \xi-AV(\xi)|\geq  {1\over \tau}|\dot \xi-AV(\xi)|\cdot |DK(\xi)|\geq {1\over \tau}|(\dot \xi-AV(\xi))\cdot DK(\xi)|={1\over \tau}|\dot w(t)|
$$
for $w(t)=K(\xi (t))$ and $t\in  [t_1,t_2]\cup [t_3,t_4]$,  we have that 
\begin{align*}
E_1+E_2&\geq \overline H_A(p)(t_2-t_1+t_4-t_3)+{1\over 4\tau^2}\left(\int_{t_1}^{t_2}|\dot w(t)|^2\,dt+\int_{t_3}^{t_4}|\dot w(t)|^2\,dt\right)\\
&\geq \overline H_A(p)(t_2-t_1+t_4-t_3)+{1\over 4\tau^2}\left({\delta^2 \over t_4-t_3}+{\delta^2 \over t_2-t_1}\right)\\
&\geq {2{\delta}\over \tau}\sqrt{\overline H_A(p)}.
\end{align*}
   However,  if we travel from  $\xi(t_1)$ to $\xi(t_4)$ along the route $\gamma(s)=\gamma_{\delta}(sA)$, the cost is  at most ${T_{\delta}\over A}\overline H_A(p)<{2{\delta}\over \tau}\sqrt{\overline H_A(p)}$.  This contradicts to the assumption that $\xi$ is a minimizing trajectory.  Hence our  above claim holds. 

Now choose $\ep_0>0$ such that $\ep_{0}^{2}+\overline M\ep_0<\bar c$ for $\overline M=\max_{\Bbb T^2}|V|$.   Owing to Lemma \ref{bend},  there exists $A_0$ such that if $A\geq A_0$, then
$$
\lambda_{p,A}\leq {\ep_0\over 2} A
$$
for any unit vector $p$.  

\medskip

\nit {\bf Claim 2:}   Assume  that $p_A\in S_A=\{\alpha_A=1\}$ has a rational outward normal vector. Then $p_A$ is a linear point of $S_A$  if $A\geq A_0$. 

It is equivalent to proving that  $\bar p={p_A\over |p_A|}$ is a linear point of  the level curve $\{\alpha_A(p)={1\over |p_A|}\}$.  We argue by contradiction.  If not,  by Theorem \ref{foliation} and (3) of Lemma \ref{basic},  $\overline H_A$ is strictly convex at $\lambda_{\bar p, A}\bar p$ and the associated projected Mather set $\mathcal{M}_{\lambda_{\bar p, A}\bar p}$ is the whole Torus.    Due to Claim 1,  we must have that 
$$
\overline H_A(\lambda_{\bar p, A}\bar p)\geq \bar c A^2.
$$
Since $\overline H_A(p)\leq |p|^2+A\overline M |p|$,  we have that $\lambda_{A,\bar p}>\ep_0 A$.  This contradicts to the choice of $A$.  
Therefore, the above claim holds and the result of this theorem follows.

\medskip

{\bf Case 2:}  Next we consider the case when $\dot\xi = V(\xi)$ has an unbounded  periodic orbit  $\eta:\Rset\to \Rset^2$, i.e., $\eta (T_0)-\eta(0)\in  \Zset^2\backslash\{(0,0)\}$ for some $T_0>0$.  Denote   $q_0$ as a rotation vector of $\eta$.  Clearly,  $q_0$ is a rational vector.   Since $\int_{\Bbb T^2}V\,dx=0$, there also exists an unbounded periodic orbit $\dot {\tilde \eta} (t)=V(\tilde \eta(t))$ with a rotation vector $-cq_0$ for some $c>0$.  See the following Figure 3.


\begin{figure}[h]
\begin{center}
\begin{tikzpicture}

\draw plot [smooth] coordinates {(-2,0) (-1,1.5) (0,1) (1,0.8) (2,2) (3,1.8) (4,3)};
\draw (0,0.8) node {$\eta$};
\draw[->] (-1,1.5)--(-0.95,1.51);
\draw[->] (1,0.8)--(1.05,0.82);

\draw plot [smooth] coordinates {(-2,-2) (-1,-0.3) (0,-1) (1,-1) (2,0) (3,-0.2) (4,1)};
\draw (0,-1.2) node {$\tilde \eta$};
\draw[<-] (-1,-0.3)--(-0.95,-0.29);
\draw[<-] (1,-1)--(1.05,-0.97);

\draw[dashed] plot [smooth] coordinates {(-2,0) (-1,2.5) (2,3) (3,3.5) (4,3)};
\draw (-1,2.8) node {$\xi_m$};
\draw (-2,-0.2) node {$x_m$};
\draw (4.8,2.8) node {$x_m+aq_0$};
\end{tikzpicture}
\caption{Unbounded periodic orbits $\eta$ and $\tilde \eta$}
\end{center}
\end{figure}
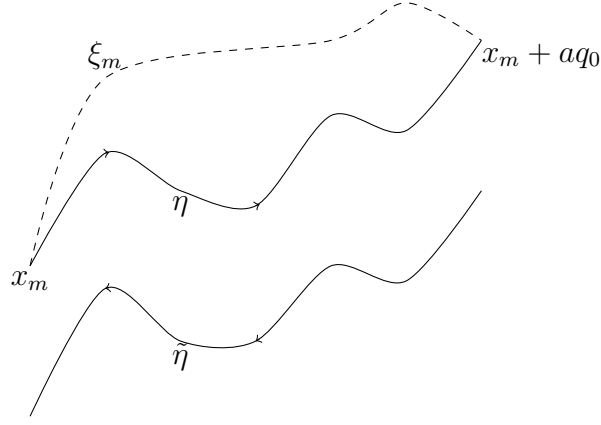

\medskip

\nit {\bf Claim 3}:  Choose $p_A\in S_A$  such that the unit outward normal vector  at $p_A$ is   $q_0\over |q_0|$. Then when $A$ is large enough,  $p_A$ is a linear point of $S_A$.  This is consistent with the last statement in Remark \ref{rem:FlatPiece}:  when $A$ is very large,  we expect the shear structure to dominate the flame propagation. 

It is equivalent to proving that  $\bar p={p_A\over |p_A|}$ is a linear point of the level curve $\{\alpha_A(p)={1\over |p_A|}\}$ for sufficiently large $A$.   We argue by contradiction.   If not,  then by (3) of Lemma \ref{basic},  there exist a sequence $A_m\to +\infty$ as $m\to +\infty$ and $|p_m|=1$ such that  $\overline H_{A_m}(\lambda_mp_m)$ is  strictly convex near $\lambda_mp_m$. Here $\lambda_m>0$ is the unique number  satisfying (Lemma \ref{basic})
$$
\alpha_{A_m}(p_m)={1+\overline H_{A_m}(\lambda_mp_m)\over \lambda_m}.
$$
Therefore, by Theorem \ref{foliation}, the associated projected Mather set $\mathcal {M}_{\lambda_mp_m}$  is the whole torus. So there exists a unique periodic $C^1$ solution  $v_m$ (up to additive constants) to 
$$
|\lambda_mp_m+Dv_m|^2+A_mV\cdot (\lambda_mp_m+Dv_m)=\overline H_{A_m}(\lambda_mp_m) \quad \text{in  $\Rset^2$}.
$$
Let $T_0$ and $\tilde T_0$ be the minimal period of $\eta$ and $\tilde \eta$ respectively. Then $q_0={\eta(T_0)-\eta(0)\over T_0}$ and $-cq_0={\tilde \eta(T_0)-\tilde \eta(0)\over T_0}$. Taking integration along $\eta$ and $\tilde \eta$,  we obtain that
$$
{1\over T_0}\int_{0}^{T_0}|\lambda_mp_m+Dv_m(\eta(s))|^2\,ds+A_mq_0\cdot \lambda_mp_m=\overline H_{A_m}(\lambda_mp_m)
$$
and
$$
{1\over \tilde T_0}\int_{0}^{\tilde T_0}|\lambda_mp_m+Dv_m(\tilde \eta(s))|^2\,ds-cA_m q_0\cdot \lambda_mp_m=\overline H_{A_m}(\lambda_mp_m).
$$
  Accordingly, without loss of generality,  we may assume that for all $m\geq 1$,
$$
\max_{s\in \Rset}|\lambda_mp_m+Dv_m(\eta(s))|\geq  \sqrt {\overline H_{A_m}(\lambda_mp_m)}.
$$
So there exists $x_m\in  \eta(\Rset)\cap [0,1]^n$ such that 
\be\label{lower}
|\lambda_mp_m+Dv_m(x_m)|\geq  \sqrt {\overline H_{A_m}(\lambda_mp_m)}. 
\ee
Since the projected Mather set $\mathcal {M}_{\lambda_mp_m}$  is the whole torus and the unit outward normal vector of $\{\overline H_{A_m}=\overline H_{A_m}(\lambda_mp_m)\}$ at $\lambda_mp_m$ is also  $q_0\over |q_0|$,  by (\ref{directionvector}),  we may find  a periodic minimizing trajectory  $\xi_m:\Rset\to  \Rset^2$ such that $\xi_m(0)=x_m$ and $\xi_m(t_m)=x_m+a_0q_0$ (see Figure 3).   Here  $t_m>0$ is the minimal period of $\xi_m$ and $a_0>0$ is the smallest positive number such that $a_0q_0\in  \Zset^2$.  Note that $\eta(T_0)-\eta(0)=a_0q_0$ as well.  Moreover, by (\ref{graphsupport}), 
\be\label{chara}
\dot \xi_m=2(\lambda_mp_m+Dv_m(\xi_m))+A_mV(\xi_m)
\ee
and
$$
{d(\dot \xi_m(s)-A_mV(\xi_m(s)))\over ds}=-(\dot \xi_m-A_mV(\xi_m))A_mDV(\xi_m).
$$
Since $\xi_m$ is an absolutely minimizing trajectory,   we must have that
\be{}\label{error}
{T_0\over A_m}\overline H_{A_m}(\lambda_mp_m)\geq \int_{0}^{t_m}{1\over 4}|\dot \xi_m(s)-A_m V(\xi_m)|^2\,ds+t_m\overline H_{A_m}(\lambda_mp_m).
\ee
The left hand side of the above is the cost of traveling along the route $\gamma(s)=\eta(sA_m)$  from $x_m$ to $x_m+aq_0$.  So $t_m\leq {T_0\over A_m}$.  
Consider
$$
w_m(s)=\xi_m\left({s\over A_m}\right).
$$
Then $w_m$ is a periodic curve with a minimal period $A_mt_m\leq T_0$, 
\be\label{equal}
{1\over 4} |\dot w_m-V(w_m)|^2+{1\over 2}V(w_m)\cdot (\dot w_m-V(w_m))={\overline H_{A_m}(\lambda_mp_m)\over A_{m}^{2}}
\ee
and
$$
{d(\dot w_m(s)-V(w_m(s)))\over ds}=-(\dot w_m(s)-V(w_m(s))\cdot DV(w_m(s)).
$$
Hence there exists a constant $\theta_0>0$ depending only on $V$ such that
\be\label{control}
{\min_{s\in \Rset}|\dot w_m(s)-V(w_m(s))|\over \max_{s\in \Rset}|\dot w_m(s)-V(w_m(s))|}>\theta_0.
\ee
Owing to (\ref{error}),  we obtain that
\be\label{integral}
\int_{0}^{A_mt_m}|\dot w_m(s)-V(w_m(s))|^2\,ds\leq {4T_0\overline H_{A_m}(\lambda_mp_m)\over A_{m}^{2}}.
\ee
Owing to  (\ref{equal}) and  Lemma \ref{bend},   $\max_{s\in \Rset}|\dot w_m(s)|$ is uniformly bounded.   Since $w_m(A_mt_m)-w_m(0)=a_0q_0$,  it is clear that 
\[
\liminf_{m\to +\infty}A_mt_m>0.
\]
Now combining (\ref{integral}), (\ref{control}) and  Lemma \ref{bend},  it is not hard to show that  
\[
\lim_{m\to +\infty}A_mt_m=T_0
\] 
and
\be\label{converge}
\lim_{m\to  +\infty}w_m(s)=\eta(s)   \quad \text{uniformly in $C^1(\Rset^1)$}.
\ee
Write $c_m=\max_{s\in \Rset}|\dot w_m(s)-V(w_m(s))|$.   Note that
$$
\dot w_m(s)={2(\lambda_mp_m+Dv_m(w_m(s)))\over A_m}+V(w_m(s)).
$$
and by (\ref{equal}), 
$$
V(w_m)\cdot {\dot w_m-V(w_m)\over c_m}={2\overline H_{A_m}(\lambda_m p_m)\over A_{m}^{2}c_m}-{1\over 2c_m}|\dot w_m-V(w_m)|^2.
$$
Due to (\ref{lower}) and (\ref{chara}),  $c_mA_m\geq 2\sqrt {\overline H_{A_m}(\lambda_mp_m)}$. Combining with  (\ref{converge}) and  Lemma \ref{bend},  we have that
\be\label{one-direction}
\lim_{m\to +\infty} V(w_m(s))\cdot {\dot w_m(s)-V(w_m(s))\over c_m}=0 \quad \text{uniformly in $\Rset^1$}.
\ee
Note that 
$$
{dK(w_m(s))\over ds}=DK(w_m(s))\cdot (\dot w_m(s)-V(w_m(s)). 
$$
Taking integration from 0 to $t_mA_m$, due to periodicity,  we have that
$$
\int_{0}^{t_mA_m}DK(w_m(s))\cdot \left({\dot w_m(s)-V(w_m(s))\over c_m}\right)\,ds=0.
$$
By   (\ref{control}),  ${|\dot w_m(s)-V(w_m(s))|\over c_m}\in [\theta_0,  1]$.   Combining with (\ref{one-direction}),  by sending $m\to +\infty$,  we obtain that
$$
\int_{0}^{T_0}a(s)|DK(\eta(s))|\,ds=0
$$
for some $a(t)>0$.   This is a contradiction.  So our claim holds.

Combining Case 1 and Case 2,  we obtain the desired result.

\end{proof}

\begin{lem}  Let $\overline H$ be the effective Hamiltonian of 
$$
|p+Dv|^2+V(x)\cdot (p+Dv)=\overline H(p).
$$
Then for $|p|\geq  \theta>0$,  there exists $\mu_\theta>0$ depending only on $\theta$ and $V$ such that
$$
\min_{q\in \partial \overline H(p)} q\cdot p \geq \overline H(p)+\mu_\theta.
$$
\end{lem}

\begin{proof}  
This follows easily from a compactness argument,  $\overline H(0)=0$, $\overline H(p)\geq |p|^2$ and the strict convexity of $\overline H$ along the radial direction (Theorem \ref{EGone}).
\end{proof}

Due to the simple equality  $\overline H(p)={\overline H_A(Ap)\over A^2}$,  we immediately derive the following corollary.  

\begin{cor}\label{strict}If $|p|\geq  \theta A$,  then
$$
\min_{q\in \partial \overline H_A(p)} q\cdot p\geq \overline H_A(p)+\mu_\theta A^2.
$$
 \end{cor}

\begin{lem}\label{bend} For $|p|=1$ and $A\geq 1$,  denote $\lambda_{p,A}$ such that
$$
\alpha_A(p)={\overline H_A(\lambda_{p,A}p)+1\over \lambda_{p,A}}.
$$
Then
$$
\lim_{A\to +\infty} {\max_{ |p|=1}\lambda_{p,A}\over A}=\lim_{A\to +\infty}{\max_{|p|=1}\overline H_A(\lambda_{p,A}p)\over A^2}=0.
$$
\end{lem}

\begin{proof} 
Since $\overline H_A(p)\leq |p|^2+A \overline M |p|$ for $\overline M=\max_{\Tset^2}|V|$,
the second limit holds true immediately once we prove the validity of the first limit.   

We prove the first limit by contradiction.  If not,  then there exists a sequence $A_m\to +\infty$ as $m\to +\infty$ and $|p_m|=1$ such that for $\lambda_m=\lambda_{p_m,A_m}$,  
$$
\lim_{m\to  +\infty}  { \lambda_{m}\over {A_m}}=b_0>0.
$$
So by Lemma \ref{basic},   there is $q_m\in   \partial \overline H_{A_m}(\lambda_m p_m)$ such that
$$
q_m\cdot \lambda_mp_m=\overline H_A(\lambda_mp_m)+1.
$$
This contradicts to Corollary \ref{strict} when $m$  is large enough.  
\end{proof}

Finally,  we prove  Theorem \ref{main3}.  

\begin{proof}[{\bf Proof of Theorem \ref{main3}}]  It suffices to show that there exists a unit vector $p_0$ such that $s p_0$ is a linear point of  $\{\overline H_A=\overline H_A(s p_0)\}$  for any $s>0$.  

\medskip

{\bf (1)}  Assume that $V$ is the shear flow, i.e.  $V=(v(x_2),0)$.   
 Without loss of generality,  we omit the dependence on $A$.  
 It is easy to see that $\overline H_A(p)$ has the following explicit formulas:  for $p=(p_1,p_2)\in  \Rset^2$, 
$$
\overline H(p)=|p_1|^2+h(p_1, p_2)
$$
and $h:\Rset^2\to  \Rset$ is given by 
$$
\begin{cases}
h(p)=M(p_1)=\max_{y\in  \Tset} p_1v(y)   \qquad &\text{if $|p_2|\leq \int_{0}^{1}\sqrt{M(p_1)-p_1v(y)}\,dy$}\\
|p_2|=\int_{0}^{1}\sqrt{h(p)-p_1v(y)}\,dy \qquad &\text{otherwise}. 
\end{cases}
$$
Hence   $\overline H$ is linear near the point $p=(s ,0)$ as long as $s\ne 0$.  

\medskip

{\bf (2)} Now  let $V=(-K_{x_2}, K_{x_1})$ for $K(x)=\sin 2\pi x_1 \sin 2\pi x_2$.  Recall that the cell problem is
\be\label{cell}
|p+Dv|^2+AV\cdot (p+Dv)=\overline H_A(p)\geq |p|^2.
\ee
The proof for the cellular flow case of Theorem \ref{main3} follows directly from the result of the following proposition.

\end{proof}

\begin{prop}\label{prop:M-2d}
Fix $s>0$ and let $Q=(s,0) \in \Rset^2$. If $A \neq 0$, then $Q$ is a linear point of the level set $\{\overline{H}_A=\overline H_A(Q)\}$.
\end{prop}

\begin{proof}
Let $\mathcal{M}_Q$ be the projected Mather set at the point $Q$.  By symmetry,  it is easy to see that  $\partial \overline H(Q)$ is parallel to (1,0).   Then due to Theorem \ref{foliation},  it suffices to show that  
\begin{equation}\label{AQ-cl}
\mathcal{M}_Q \cap \{y \in \Tset^2\,:\,y_2=0\} = \emptyset.
\end{equation}

{\bf Step 1:}  We claim that there is a viscosity solution $v$ to (\ref{cell}) which satisfies that $v(y_1,y_2)=v(y_1,-y_2)$.  In fact,  let us now look at the discounted approximation of \eqref{cell} with $p=Q$. For each $\ep>0$, consider
\begin{equation}\label{EQ-ep}
\ep v^\ep+ |Q+Dv^\ep|^2 + AV\cdot(Q+Dv^\ep) = 0 \quad \text{in} \ \Tset^2,
\end{equation}
which has a unique viscosity solution $v^\ep \in C^{0,1}(\Tset^2)$.
By the fact that $Q=(s,0)$ and the special structure of $V$, 
it is clear that $(y_1,y_2) \mapsto v^\ep(y_1,-y_2)$ is also a solution to the above. 
Therefore, $v^\ep(y_1,y_2)=v^\ep(y_1,-y_2)$ for all $(y_1,y_2)\in \Tset^2$.  Clearly,  any convergent subsequence of   $v^\ep - v^\ep(0)$ tends to a $v$ which is  a solution of \eqref{cell} and is even in the $y_2$ variable.    We would like to point out that a recent result of Davini, Fathi, Iturriaga and Zavidovique \cite{DFIZ} (see also Mitake and Tran \cite{MT14}) gives the convergence of the full sequence $v^\ep - v^\ep(0)$  as $\ep \to 0$. 

\medskip

{\bf Step 2:}  Assume by contradiction that (\ref{AQ-cl}) is not correct.  Suppose that 
$$
(\mu_0,0)\in \mathcal{M}_Q \cap \{y \in \Tset^2\,:\,y_2=0\}.
$$
Then $v$ is differentiable at $(\mu_0,0)$ and  $v_{x_2}(\mu_0,0)=0$.   Due to  (\ref{graphsupport}),  the flow-invariance of the Mather set and the Euler-Lagrangian equation,  it is easy to see that
$$
\{y \in \Tset^2\,:\,y_2=0\}\subset \mathcal{M}_Q.
$$
Hence $v$ is $C^1$ along the $y_1$ axis and 
\begin{equation}\label{eq:M1}
v_{y_2}(y_1,0)=0 \quad \text{for all} \ y_1 \in \Tset.
\end{equation} 
Set $w(y_1)=v(y_1,0)$. Plug this into the equation \eqref{cell} of $v$ with $y_2=0$ and use \eqref{eq:M1} to get that
\begin{equation}\label{eq:M2}
|s+w'|^2 - A (s+w')\sin (2\pi y_1) =\overline{H}_A(Q)\geq s^2  \quad \text{in} \ \Tset. 
\end{equation}
Clearly, $s+w'(y_1) \neq 0$ for all $y_1 \in \Tset$ in light of \eqref{eq:M2}. 
Note further that $w'\in C(\Tset)$ and
\[
\int_0^1 (s+w'(y_1))\,dy_1 = s >0.
\]
Thus, $s+w'>0$ in $\Tset$ and for all $y_1 \in \Tset$,
\[
s+w'(y_1)  = \frac{1}{2} \left( A\sin (2\pi y_1) + \sqrt{A^2 \sin^2 (2\pi y_1) + 4\overline{H}_A(Q)} \right).
\]
Integrate this over $\Tset$ to deduce that
\begin{align*}
s=\int_0^1 (s+w'(y_1))\,dy_1&= \int_0^1  \frac{1}{2} \left( A\sin (2\pi y_1) + \sqrt{A^2 \sin^2 (2\pi y_1) + 4\overline{H}_A(Q)} \right)\,dy_1\\
&= \int_0^1  \frac{1}{2}  \sqrt{A^2 \sin^2 (2\pi y_1) + 4\overline{H}_A(Q)} \,dy_1 \geq \sqrt{\overline{H}_A(Q)} \geq s.
\end{align*}
Therefore, all inequalities in the above must be equalities.
In particular, the second last inequality must be an equality, which yields that $A=0$.
\end{proof}
\begin{rmk} Theorem \ref{foliation} is not really necessary to get the above proposition.  In fact,  using the same argument,  we can derive that the Aubry set has no intersection with the $y_1$ axis.  Then by \cite{FS04},   there is a strict subsolution to  \eqref{cell} near the $y_1$ axis.   The linearity of $\overline H$ near $Q$ will follow from some elementary calculations. 

\end{rmk}

\bibliographystyle{plain}

\begin{thebibliography}{99}

\bibitem{Alan}

A. R. Kerstein,  J. R.  Mayo,  {\em Propagation anomaly of Huygens fronts in turbulence}, submitted. 

\bibitem{KTW}  

A. R. Kerstein, W. T. Ashurst, and F. A. Williams, {\em Field equation for interface propagation in an unsteady homogeneous  flow field},  Phys. Rev. A 37, 2728 (1988).

\bibitem{B1}
V. Bangert,   
{\em Mather Sets for Twist Maps and Geodesics on Tori},   Dynamics Reported, Volume 1. 

\bibitem{B2}
V. Bangert,  
{\em Geodesic rays, Busemann functions and monotone twist maps}, Calculus of Variations and Partial Differential Equations
January 1994, Volume 2, Issue 1, pp 49--63.  

\bibitem{BE}
M.  Bardi, L. C. Evans,  
{\em On Hopf's formulas for solutions of Hamilton-Jacobi equations}. Nonlinear Anal. 8 (1984), no. 11, 1373--1381.

\bibitem{Bialy}

M.L. Bialy,  {\em Rigidity for periodic magentic field}, Ergodic Theory and Dynamical Systems,  Vol. 20,  Issue 06,  2000, pp 1619--1626.


\bibitem{C}

M. J. Carneiro, 
{\em On minimizing measures of the action of autonomous Lagrangians},
Nonlinearity 8 (1995) 1077--1085.


\bibitem{DFIZ}

A. Davini, A. Fathi, R. Iturriaga, and M. Zavidovique, {\em Convergence of the solutions of the discounted equation}, to appear in Invent. Math, 2016.


\bibitem{W-E}
W. E,
\emph{Aubry-Mather theory and periodic solutions of the forced Burgers equation},
Comm. Pure Appl. Math. 52 (1999), no. 7, 811--828. 

\bibitem{EMS}

P. Embid, A. Majda and P. Souganidis, {\em Comparison of turbulent flame speeds from complete averaging and the G-equation}, Phys. Fluids 7(8) (1995), 2052--2060.


\bibitem{EG}
L. C. Evans, D. Gomes, 
\emph{Effective Hamiltonians and Averaging for
Hamiltonian Dynamics. I}, Arch. Ration. Mech. Anal. 157 (2001), no. 1, 1--33.

\bibitem{E-lecture}
L. C. Evans,
\emph{Weak KAM theory and partial differential equations},
 Calculus of variations and nonlinear partial differential equations, 123--154,
Lecture Notes in Math., 1927, Springer, Berlin, 2008.

\bibitem{F}
A. Fathi, 
\emph{The weak KAM theorem in Lagrangian dynamics},
Cambridge University Press (2004).

\bibitem{FS04}
A. Fathi, A. Siconolfi,
{\em Existence of $C^1$ critical subsolutions of the Hamilton-Jacobi equation}
Invent. Math. 155 (2004), no. 2, 363--388. 


\bibitem{G2003}
D. Gomes, {\em Perturbation theory for viscosity solutions of Hamilton-Jacobi equations and stability of Aubry-Mather sets}, SIAM J. Math. Anal. 35 (2003), no. 1, 135--147.

\bibitem{H}
E. Hopf, 
{\em  Closed Surfaces Without Conjugate Points}, Proc. Nat. Acad. of Sci. 34, 1948.


\bibitem{LPV}  
P. L. Lions,  G. C. Papanicolaou, and S. R. S.  Varadhan,  
{\em  Homogenization of Hamilton-Jacobi equation},  Unpublished preprint,  1987.  


\bibitem{LTY1} 
S. Luo, H. V. Tran, and Y. Yu,
{\em Some inverse problems in periodic homogenization of Hamilton-Jacobi equations}, to appear in Arch. Rational. Mechanics. Anal.


\bibitem{MS_94} A. Majda,  P. E. Souganidis,
{\em Large scale front dynamics for turbulent reaction--diffusion
equations with separated velocity scales},  {Nonlinearity},  7(1994), pp 1--30.

\bibitem{MS}
D. Massart,   A. Sorrentino,   
{\em Differentiability of Mather's average action and integrability on closed surfaces},  
Nonlinearity 24 (2011), no. 6, 1777--1793. 

\bibitem{MT14} 
H. Mitake, H. V. Tran,
{\em Selection problems for a discounted degenerate viscous Hamilton--Jacobi equation},
arXiv:1408.2909 [math.AP], submitted.



\bibitem{RGC}
R.  Ruggiero, J.  B. Gomes and M.  J. D. Carneiro,  
{\em Hopf conjecture holds for analytic, K-basic  Finsler $2$-tori without conjugate points},  
Geometry and Foliations 2013,  Komaba, Tokyo, Japan.

\bibitem{XY} 
J. Xin,  Y. Yu,  
{\em  Sharp asymptotic growth laws of turbulent flame speeds in cellular flows by inviscid Hamilton-Jacobi models},  
Annales de l'Institut Henri Poincare, Analyse Nonlineaire, 30(6), pp. 1049--1068, 2013. 



\bibitem{Z}
N. Zinovev,   
{\em Examples of Finsler metrics without conjugate points: metric of revolution},  
St. Petersburg Math. J.  Vol. 20 (2009), No. 3, Pages 361--379.




\end{thebibliography}

\end{document}